\documentclass[12pt]{article}
\usepackage{a4}
\usepackage{amsthm}
\usepackage{amsfonts}
\usepackage{amssymb}
\usepackage{cite}
\usepackage{epsfig}
\usepackage[colorlinks]{hyperref}
\usepackage[intlimits]{amsmath}
\newtheorem{theorem}{Theorem}
\newtheorem{lemma}[theorem]{Lemma}
\newtheorem{corollary}[theorem]{Corollary}

\newtheorem{problem}{Problem}
\newcommand{\rt}{\right}
\newcommand{\lt}{\left}

\renewcommand{\epsilon}{\varepsilon}
\newcommand{\s}{\sigma}
\renewcommand{\a}{\alpha}
\renewcommand{\l}{\lambda}
\def\Aut{{\rm Aut}}
\def\AA{{\cal A}}
\def\FF{{\cal F}}
\def\II{{\cal I}}
\def\RRR{{\cal R}}
\def\NN{{\mathbb N}}
\def\RR{{\mathbb R}}
\def\dif{{\rm d}}
\def\oS{{\overline{S}}}

\def\cQ{{\mathcal{Q}}}
\begin{document}
\title{Finitely forcible graphons and permutons\thanks{The~work leading to this invention has received funding from the European Research Council under the European Union's Seventh Framework Programme (FP7/2007-2013)/ERC grant agreement no.~259385.}}
\author{Roman Glebov\thanks{Department of Mathematics, ETH, 8092 Zurich, Switzerland. E-mail: {\tt roman.l.glebov@gmail.com}. Previous affiliation: Mathematics Institute and DIMAP, University of Warwick, Coventry CV4 7AL, UK.}\and
        Andrzej Grzesik\thanks{Theoretical Computer Science Department, Faculty of Mathematics and Computer Science, Jagiellonian University, ul. Prof. St.  Lojasiewicza 6, 30-348 Krakow, Poland. E-mail: {\tt Andrzej.Grzesik@uj.edu.pl}. Partially supported by NCN grant 2011/01/N/ST1/02341.}\and
        Tereza Klimo{\v s}ov\'a\thanks{Mathematics Institute and DIMAP, University of Warwick, Coventry CV4 7AL, UK. E-mail: {\tt t.klimosova@warwick.ac.uk}.}\and
        Daniel Kr\'al'\thanks{Mathematics Institute, DIMAP and Department of Computer Science, University of Warwick, Coventry CV4 7AL, UK. E-mail: {\tt d.kral@warwick.ac.uk}.}}

\date{}
\maketitle

\begin{abstract}
We investigate when limits of graphs (graphons) and permutations (permutons)
are uniquely determined by finitely many densities of their substructures,
i.e., when they are finitely forcible. Every permuton can be associated
with a graphon through the notion of permutation graphs. We find permutons
that are finitely forcible but the associated graphons are not. We also
show that all permutons that can be expressed as a finite combination of
monotone permutons and quasirandom permutons are finitely forcible, which
is the permuton counterpart of the result of Lov\'asz and S\'os for graphons.
\end{abstract}


\section{Introduction}

Analytic objects associated with convergent sequences of combinatorial objects
have recently been attracting a significant amount of attention. This line of
research was initiated by the theory of limits of dense graphs~\cite{bib-borgs08+,bib-borgs+,bib-borgs06+,bib-lovasz06+},
followed by limits of sparse graphs~\cite{bib-bollobas11+,bib-elek07},
permutations~\cite{bib-hkmrs1,bib-hkms2}, partial orders~\cite{bib-janson11}, etc.
Further details can also be found in a recent monograph on graph limits of Lov\'asz~\cite{bib-lovasz-book}.
Analytic methods applied to such limit objects led to results
in many areas of mathematics and computer science, most importantly,
extremal combinatorics~\cite{bib-flag1, bib-flag2, bib-flagrecent, bib-flag3, bib-flag4, bib-flag5, bib-flag6, bib-flag7, bib-flag8, bib-flag9, bib-flag10, bib-razborov07, bib-flag11, bib-flag12}
and property testing~\cite{bib-hoppen-test,bib-lovasz10+}.
In particular, Lov\'asz and Szegedy in~\cite{bib-lovasz10+}
characterize testable properties and parameters of graph limits (graphons) and
they identify a class of testable graph properties, so-called flexible properties.
In this paper, we focus on a question when the limit analytic object
is uniquely determined by finitely many densities of substructures.
This phenomenon is known as finite forcibility.

In fact, questions of this kind are closely related to quasirandomness and
they were studied well before the theory of limits of combinatorial objects emerged.
For example,
the results on quasirandom graphs from work of Chung, Graham and Wilson~\cite{bib-chung89+}, R\"odl~\cite{bib-rodl} and Thomason~\cite{bib-thomason, bib-thomason2}
imply that the homomorphic densities of $K_2$ and $C_4$ guarantee that densities of all subgraphs behave as in the random graph $G_{n,1/2}$.
In the language of graphons (the limit structure for graphs),
this result asserts that the graphon identically equal to $1/2$
is finitely forcible by densities of $4$-vertex subgraphs.
A similar result on permutations, which was originally raised as a question by Graham,
was proven by the last author and Pikhurko~\cite{bib-kral12+}
who exploited the analytic view of permutation limits.

Let us now give motivation for the concepts.
The result on finite forcibility of the graphon identically equal to $1/2$
was generalized by Lov\'asz and S\'os~\cite{bib-lovasz08+}
who proved that any stepwise graphon is finitely forcible.
These results were further extended by Lov\'asz and Szegedy~\cite{bib-lovasz11+}
who also gave several conditions when a graphon is not finitely forcible.

In this paper, we are interested in finite forcibility of graph limits (graphons) and
permutation limits (permutons). We start with proving an analogue of the result
of Lov\'asz and S\'os~\cite{bib-lovasz08+} for permutons,
which is stated as Corollary~\ref{cor-force}.
We then focus
on finite forcibility of permutons with infinite recursive structure, and
on the interplay between finite forcibility of permutons and graphons,
partly motivated by Question~11 from~\cite{bib-lovasz11+}.

A graph can be associated with a permutation in the following way: the vertices of the graph
correspond to the elements of the permutation and two of them are joined by an edge
if they form an inversion.
Along the same lines, a graphon can be associated with a permuton.
For example, the graphon associated with the limit of random permutations
is one of the candidates for a finite forcible ``2-dimensional'' graphon (see Section~\ref{sec:concl} for more details).
Unfortunately, we show that there exist finitely forcible permutons such that
the associated graphons are not finitely forcible.
In Section~\ref{sect-infinite},
we find two families of finitely forcible permutons,
which have infinite recursive structure.
Then, we show that the associated graphons are not finitely forcible in Section~\ref{sect-graphon},
where we prove this result for all graphons with the recursive structure analogous
to that of permutons from Section~\ref{sect-infinite}.
Still, we believe that some permutons leading to ``2-dimensional'' graphons
can be finitely forcible and we mention two particular cases at the end of the paper.
Also let us remark that
the methods we use in Section~\ref{sect-infinite}
were subsequently extended by Jan Volec and two of the authors~\cite{bib-comp} to resolve
Conjecture 9 from~\cite{bib-lovasz11+} on the compactness of finitely forcible graphons and
they were also used to resolve Conjecture 10 on the dimension of such graphons in~\cite{bib-inf}.

\section{Notation}
\label{sect-notation}

In this section, we introduce concepts related to graphs and permutations and
their limits used in the paper.
We start with the slightly simpler notion of permutation limits.

\subsection{Permutations and permutons}
\label{subs-permuton}

The theory of permutation limits was built by Hoppen, Kohayakawa, Moreira,
R\'ath and Sampaio in~\cite{bib-hkmrs1,bib-hkms2}. Here, we follow the analytic
view of the limit as used in~\cite{bib-kral12+}, which also appeared
in an earlier work of Presutti and Stromquist~\cite{bib-wrs}.

A {\em permutation of order} $n$ is a bijective mapping from $[n]$ to $[n]$,
where $[n]$ denotes the set of integers from $1$ to $n$. The order of a permutation $\pi$
is also denoted by $|\pi|$.
The set of all permutations of order $n$ is denoted by $S_n$.
In what follows, we identify a sequence of $n$ different integers $a_1\ldots a_n$ between $1$ and $n$
with a permutation $\pi$ by setting $\pi(i)=a_i$. For example, the identity permutation of order $4$
is denoted by $1234$.

If $\pi$ is a permutation of order $n$,
a {\em subpermutation\/} induced by $1\le i_1<\ldots<i_k\le n$ in $\pi$
is a permutation $\sigma$ of order $k$ such that $\sigma(j)<\sigma(j')$ if and only if $\pi(i_j)<\pi(i_{j'})$.
For example, the subpermutation of $7126354$ induced by $3,4,6$ is $132$.
Subpermutations are more commonly referred to as {\em patterns}
but we decided to use the term subpermutation in the paper
to be consistent with the analogous concept for graphs as well as
with previous work on permutation limits.
A {\em density\/} $d(\sigma,\pi)$ of a permutation $\sigma$ of order $k$
in a permutation $\pi$ of order $n$ is the number of $k$-tuples
inducing $\sigma$ in $\pi$ divided by ${n\choose k}$.
Conveniently, we set $d(\sigma,\pi)=0$ if $k>n$.

An infinite sequence $(\pi_i)_{i\in\NN}$ of permutations with $|\pi_i|\to\infty$
is {\em convergent} if $d(\sigma,\pi_i)$ converges for every permutation $\sigma$.
We will see that one can associate with every convergent sequence of permutations the following analytic
object:
a {\em permuton} is a probability measure $\mu$ on the $\sigma$-algebra $\AA$ of Borel sets
of the unit square $[0,1]^2$ such that $\mu$ has {\em uniform marginals},
i.e.,
$\mu\lt(\lt[\alpha,\beta\rt]\times [0,1]\rt)=\mu\lt([0,1]\times[\alpha,\beta]\rt)=\beta-\alpha$
for every $0\le\alpha\le\beta\le 1$. In what follows, we use $\lambda$
for the uniform measure on $\AA$.
More generally, if $A\subseteq [0,1]^2$ is a non-trivial convex polygon,
i.e., a convex polygon different from a point (however, which can be a segment),
we define $\lambda_{A}$ to be the unique probability measure on $\AA$
with support $A$ and
mass uniformly distributed inside $A$. In particular,
$\lambda_{[0,1]^2}=\lambda$.

We now describe the relation between permutons and convergent sequences of permutations.
Let $\mu$ be a permuton. For an integer $n$, one can sample $n$ points
$(x_1,y_1),\ldots,(x_n,y_n)$ in $[0,1]^2$ randomly based on $\mu$.
Because $\mu$ has uniform marginals,
the $x$-coordinates of all these points are mutually different with probability one.
The same holds for their $y$-coordinates. Assume that this is indeed the case.
One can then define a permutation $\pi$ of order $n$
based on the $n$ points $(x_1,y_1),\ldots,(x_n,y_n)$ as follows:
let $i_1,\ldots,i_n\in [n]$ be such that
$x_{i_1}<x_{i_2}<\cdots<x_{i_n}$ and define $\pi$ to be the unique bijective mapping
from $[n]$ to $[n]$ satisfying that $\pi(j)<\pi(j')$ if and only if $y_{i_j}<y_{i_{j'}}$.
We will say that a permutation $\pi$ of order $n$ obtained in the just described
way is a {\em $\mu$-random permutation\/} of order $n$.
A {\em uniformly random permutation} is a $\lambda$-random permutation,
i.e., each permutation of order $n$ is chosen with probability $1/n!$ at random.

If $\mu$ is a permuton and $\sigma$ is a permutation of order $n$,
then $d(\sigma,\mu)$ is the probability that a $\mu$-random permutation of order $n$ is $\sigma$.
We now recall the core results from~\cite{bib-hkmrs1,bib-hkms2}.
For every convergent sequence $(\pi_i)_{i\in\NN}$ of permutations,
there exists a unique permuton $\mu$ such that
\[d(\sigma,\mu)=\lim_{i\to\infty} d(\sigma,\pi_i)\mbox{ for every permutation $\s$.}\]
This permuton is the {\em limit\/} of the sequence $(\pi_i)_{i\in\NN}$.
On the other hand, if $\mu$ is a permuton and $\pi_i$ is a $\mu$-random permutation of order $i$,
then with probability one the sequence $(\pi_i)_{i\in\NN}$ is convergent and $\mu$ is its limit.

\begin{figure}
\begin{center}
\epsfbox{permforce.1} \hskip 10mm
\epsfbox{permforce.2} \hskip 10mm
\epsfbox{permforce.3} \hskip 10mm
\epsfbox{permforce.4}
\end{center}
\caption{The limits of sequences $\lt(\pi^1_i\rt)_{i\in\NN}$, $\lt(\pi^2_i\rt)_{i\in\NN}$, $\lt(\pi^3_i\rt)_{i\in\NN}$ and $\lt(\pi^4_i\rt)_{i\in\NN}$
         from Subsection~\ref{subs-permuton}.}\label{fig-permuton}
\end{figure}

We now give four examples of the just defined notions (the corresponding permutons are depicted
in Figure~\ref{fig-permuton}).
Let us consider a sequence $\lt(\pi^1_i\rt)_{i\in\NN}$ such that $\pi^1_i$ is the identity permutation of order $i$,
i.e., $\pi^1_i(k)=k$ for $k\in [i]$. This sequence is convergent and its limit is the measure $\lambda_A$
where $A=\lt\{(x,x), x\in [0,1]\rt\}$.
Similarly, the limit of a sequence $\lt(\pi^2_i\rt)_{i\in\NN}$, where $\pi^2_i$ is the permutation of order $i$
defined as $\pi^2_i(k)=i+1-k$ for $k\in [i]$, is $\lambda_B$ where $B=\lt\{(x,1-x), x\in [0,1]\rt\}$.
A little bit more complicated example is the following: the sequence $\lt(\pi^3_i\rt)_{i\in\NN}$,
where $\pi^3_i$ is the permutation of order $2i$ defined as
\[\pi^3_i(k)=\left\{
             \begin{array}{ll}
             2k-1 & \mbox{if $k\in [i]$,}\\
	     2(k-i) & \mbox{otherwise}
	     \end{array}
             \right.\;
\]
is convergent and the limit of the sequence is the measure $\frac{1}{2}\l_C+\frac{1}{2}\l_D$,
where $C=\lt\{(x/2,x), x\in [0,1]\rt\}$ and $D=\lt\{((x+1)/2,x), x\in [0,1]\rt\}$.
Next, consider a sequence $(\pi^4_i)_{i\in\NN}$ such that $\pi^4_i$ is a random permutation.
This sequence is convergent with probability one and its limit is the measure $\lambda$.

A permuton $\mu$ is {\em finitely forcible\/}
if there exists a finite set $S$ of permutations such that
every permuton $\mu'$ satisfying $d(\sigma,\mu)=d(\sigma,\mu')$ for every $\sigma\in S$
is equal to $\mu$.
For example, the following result from~\cite{bib-kral12+} asserts that the random permuton is finitely forcible with $S=S_4$.
\begin{theorem}
\label{thm-quasirandom}
Let $\mu$ be a permuton.
It holds that $d(\sigma,\mu)=1/24$ for every $\sigma\in S_4$
if and only if $\mu=\lambda$.
\end{theorem}
If $\mu$ is a permuton, then $F_{\mu}$ is the function from $[0,1]^2$ to $[0,1]$
defined as $F_{\mu}(x,y)=\mu\lt([0,x]\times [0,y]\rt)$. For example, if $\mu=\lambda$, then $F_{\mu}(x,y)=xy$.
Observe that $F_{\mu}$ is always a continuous function satisfying $F_{\mu}(\xi,1)=F_{\mu}(1,\xi)=\xi$ for every $\xi\in [0,1]$.
Furthermore, notice that $\mu\neq\mu'$ implies $F_\mu\neq F_{\mu'}$, that is, the function $F_\mu$ determines the permuton $\mu$.

The next theorem was implicitly proven in~\cite{bib-kral12+}. We include its proof for completeness.
\begin{theorem}
\label{thm-express}
Let $p(x,y)$ be a polynomial and $k$ a non-negative integer. There exist a finite set $S$ of permutations and coefficients $\gamma_\sigma$, $\sigma\in S$,
such that
\begin{equation}
\int_{[0,1]^2} p(x,y)F_{\mu}^k(x,y)\dif\lambda=\sum_{\sigma\in S}\gamma_\sigma d(\sigma,\mu)\label{eq-express}
\end{equation}
for every permuton $\mu$.
\end{theorem}

\begin{proof}
By additivity, it suffices to consider the case $p(x,y)=x^\alpha y^\beta$ for non-negative integers $\alpha$ and $\beta$.
Fix a permuton $\mu$.
Since $\mu$ has uniform marginals,
the product $x^\alpha y^\beta F_{\mu}^k(x,y)$ for $(x,y)\in [0,1]^2$ is equal to the probability that
out of $\alpha+\beta+k$ points are chosen randomly independently based on $\mu$,
the first $\alpha$ points belong to $[0,x]\times [0,1]$,
the next $\beta$ points belong to $[0,1]\times [0,y]$, and
the last $k$ points belong to $[0,x]\times [0,y]$.
So, the integral in (\ref{eq-express}) is equal to the probability that the above holds for a uniform choice of a point $(x,y)$ in $[0,1]^2$.

Since $\mu$ is a measure with uniform marginals,
a point $(x,y)$ uniformly distributed in $[0,1]^2$ can be obtained by sampling two points randomly independently based on $\mu$ and
setting $x$ to be the first coordinate of the first of these two points and $y$ to be the second coordinate of the second point.
Thus, we can consider the following random event.
Let us choose $\alpha+\beta+k+2$ points independently at random based on $\mu$
and denote by $x$ the first coordinate of the last but one point, and by $y$ is the second coordinate of the last point.
Then the integral on the left hand side of (\ref{eq-express})
is equal to the probability that
the first $\alpha$ points belong to $[0,x]\times [0,1]$,
the next $\beta$ points belong to $[0,1]\times [0,y]$, and
the following $k$ points belong to $[0,x]\times [0,y]$.
We conclude that the equation (\ref{eq-express}) holds with $S=S_{\alpha+\beta+k+2}$ and
$\gamma_\sigma$ equal to the probability that the following holds for a random permutation $\pi$ of order $\alpha+\beta+k+2$:
$\pi(i)\le\pi(\alpha+\beta+k+1)$ for $i\leq \alpha$ and for $\alpha+\beta+1\leq i \leq\alpha+\beta+k$, and
$\sigma(\pi(i))\le\sigma(\pi(\alpha+\beta+k+2))$ for $\alpha+1\leq i \leq\alpha+\beta+k$.
\end{proof}

Instead of sampling two additional points to get a random point with respect to the uniform measure $\lambda$,
we can also sample just a single point, which is a random point with respect to $\mu$.
This gives the following.

\begin{theorem}
\label{thm-express-mu}
Let $p(x,y)$ be a polynomial and $k$ a non-negative integer. There exist a finite set $S$ of permutations and coefficients $\gamma_\sigma$, $\sigma\in S$,
such that
\begin{equation}
\int_{[0,1]^2} p(x,y)F_{\mu}^k(x,y)\dif\mu=\sum_{\sigma\in S}\gamma_\sigma d(\sigma,\mu)
\end{equation}
 for every permuton $\mu$.
\end{theorem}

Let now $\oS_k$ be the set of permutations of order $k$ with one distinguished element;
we call such permutations {\em rooted\/}.
To denote rooted permutations, we add a bar above the distinguished element: e.g.,
if the second element of the permutation $2341$ is distinguished,
we write $2\overline{3}41$. Note that $\lt|\oS_k\rt|=k!\cdot k$.
If $\sigma\in\oS_k$, then $F_\mu^\sigma(x,y)$ is the probability that
the point $(x,y)$ and $k-1$ points randomly independently chosen based on $\mu$
induce the permutation $\sigma$ with the distinguished element corresponding to the point $(x,y)$.
Observe that $F_\mu(x,y)=F_\mu^{1\overline{2}}(x,y)$,
$F_\mu^{1\overline{2}}(x,y)+F_\mu^{2\overline{1}}(x,y)=x$ and
$F_\mu^{1\overline{2}}(x,y)+F_\mu^{\overline{2}1}(x,y)=y$.
A reader familiar with the concept of flag algebras developed
by Razborov~\cite{bib-razborov07} might recognize the notion of $1$-labelled flags
in the just introduced notation.

Similarly to Theorem~\ref{thm-express-mu}, the following is true.
Since the proof is completely analogous to that of Theorem~\ref{thm-express},
we decided to state the theorem without giving its proof.

\begin{theorem}
\label{thm-express-mu-general}
Let $\Sigma$ be a multiset of rooted permutations.
There exist a finite set $S$ of permutations and coefficients $\gamma_\sigma$, $\sigma\in S$,
such that
\begin{equation}
\int_{[0,1]^2} \prod_{\sigma\in\Sigma}F_{\mu}^\sigma(x,y)\dif\mu=\sum_{\sigma\in S}\gamma_\sigma d(\sigma,\mu)
\end{equation}
for every permuton $\mu$.
\end{theorem}

\subsection{Graphs and graphons}
\label{subs-graphon}

The other limit structure we consider is limits of graphs.
A {\em graph\/} is a pair $(V,E)$ where $E\subseteq {V\choose 2}$.
The elements of $V$ are called {\em vertices\/} and the elements of $E$ are called {\em edges\/}.
The {\em order\/} of a graph $G$ is the number of its vertices and it is denoted by $|G|$.
If $G$ and $G'$ are graphs, then $G\cup G'$ is the disjoint union of $G$ and $G'$ and
$G+G'$ is the graph obtained from $G\cup G'$ by adding all edges between $G$ and $G'$.
Finally, if $G$ is a graph and $U$ is a subset of its vertices,
then $G\setminus U$ is the graph obtained from $G$ by removing the vertices of $U$ and all edges containing at least one vertex from $U$.
The {\em density\/} $d(H,G)$ of $H$ in $G$ is the probability that $|H|$ randomly chosen
vertices of $G$ induce a subgraph isomorphic to $H$. If $|H|> |G|$, we set $d(H,G)=0$.

We now survey basic results related to the theory of dense graph limits as developed
in~\cite{bib-borgs08+,bib-borgs+,bib-borgs06+,bib-lovasz06+}.
A sequence of graphs $(G_i)_{i\in\NN}$ is {\em convergent\/} if the limit $d(H,G_i)$ exists for every $H$.
The associated limit object is called a {\em graphon\/}: it is a symmetric $\l$-measurable function from $[0,1]^2$ to $[0,1]$.
Here, symmetric stands for the property that $W(x,y)=W(y,x)$ for every $x,y\in [0,1]$.
If $W$ is a graphon, then a {\em $W$-random graph\/} of order $k$ is obtained by sampling $k$ random points $x_1,\ldots,x_k\in [0,1]$
uniformly and independently and joining the $i$-th and the $j$-th vertex by an edge with probability $W(x_i,x_j)$.
As in the case of permutations, we write $d(H,W)$ for the probability that a $W$-random graph of order $|H|$ is isomorphic to $H$.
For every convergent sequence $(G_i)_{i\in\NN}$ of graphs, there exists a graphon $W$ such that
$d(H,W)=\lim_{i\to\infty} d(H,G_i)$ for every graph $H$.
We call such a graphon $W$ a {\em limit} of $(G_i)_{i\in\NN}$.
On the other hand, if $W$ is a graphon, then with probability one the sequence $(G_i)_{i\in\NN}$ where $G_i$ is a $W$-random graph of order $i$
is convergent and its limit is $W$.

Unlike in the case of permutations, the limit of a convergent sequence of graphs is not unique.
For example, if $W$ is a limit of $(G_i)_{i\in\NN}$ and
$\varphi:[0,1]\to[0,1]$ is a measure preserving transformation,
then the graphon $W':=W(\varphi(x),\varphi(y))$ is also a limit of $(G_i)_{i\in\NN}$.
Let us introduce the following definition of equivalence of graphons:
two graphons $W$ and $W'$ are {\em weakly isomorphic\/} if $d(H,W)=d(H,W')$ for every graph $H$.

Finally, a graphon $W$ is {\em finitely forcible\/} if there exist graphs $H_1,\ldots,H_k$ such that
any graphon $W'$ satisfying $d(H_i,W)=d(H_i,W')$ for $i\in[k]$ is weakly isomorphic to $W$.

The densities of graphs in a graphon $W$ can be expressed as integrals using $W$.
If $W$ is a graphon and $H$ is a graph of order $k$ with vertices $v_1,\ldots,v_k$ and edge set $E$, then
\[
d(H,W)=\frac{k!}{|\Aut(H)|}\int\limits_{[0,1]^k}\prod_{v_iv_j\in E} W(x_i,x_j)\prod_{v_iv_j\not\in E} (1-W(x_i,x_j)) \dif x_1\dots\dif x_k\label{eq-densW}
\]
where $\Aut(H)$ is the automorphism group of $H$.

A permutation $\pi$ of order $k$ can be associated with a graph $G_\pi$ of order $k$ as follows.
The vertices of $G_\pi$ are the integers between $1$ and $k$ and $ij$ is an edge of $G$ if and only if
either $i<j$ and $\pi(i)>\pi(j)$, or $i>j$ and $\pi(i)<\pi(j)$.
If $(\pi_i)_{i\in\NN}$ is a convergent sequence of permutations,
then the sequence of graphs $(G_{\pi_i})_{i\in\NN}$ is also convergent.
Moreover, if two convergent sequences of permutations have the same limit,
then the graphons associated with the two corresponding (convergent) sequences of graphs are weakly isomorphic.
In this way, we may associate each permuton $\mu$ with a graphon $W_\mu$,
which is unique up to a weak isomorphism (see Figure~\ref{fig-graphon} for examples).

\begin{figure}
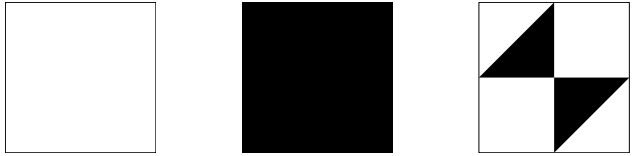

\begin{center}
\epsfbox{permforce.16} \hskip 10mm
\epsfbox{permforce.17} \hskip 10mm
\epsfbox{permforce.18}
\end{center}
\caption{The graphons associated with the first three permutons depicted in Figure~\ref{fig-permuton},
where the point (0,0) is in the lower left corner.}\label{fig-graphon}
\end{figure}

\section{Permutons with finite structure}
\label{sect-finite}

In this section, we give a sufficient condition on a permuton to be finitely forcible.
A function $f:[0,1]^2\to\RR$ is called piecewise polynomial if there exist finitely
many polynomials $p_1,\ldots,p_k$ such that
$f(x,y)\in\lt\{p_1(x,y),\ldots, p_k(x,y)\rt\}$ for every $(x,y)\in [0,1]^2$.

\begin{theorem}
\label{thm-force}
Every permuton $\mu$ such that $F_\mu$ is piecewise polynomial is finitely forcible.
\end{theorem}

\begin{proof}
Let $\mu$ be a permuton such that $F_\mu$ is piecewise polynomial,
that is, there exist polynomials $p_1,\ldots,p_k$ such that $F_\mu(x,y)\in\lt\{p_1(x,y),\ldots,p_k(x,y)\rt\}$ for every $(x,y)\in [0,1]^2$.
Let $\FF$ be the set of all continuous functions $f$ on $[0,1]^2$ such that $f(x,y)\in\lt\{p_1(x,y),\ldots,p_k(x,y)\rt\}$ for every $(x,y)\in [0,1]^2$.
The set $\FF$ is finite.
Indeed, let
\[q(x,y)=\prod_{1\le i<j\le k}\lt(p_j(x,y)-p_i(x,y)\rt)\]
and let $Q$ be the set of all points $(x,y)\in \RR^2$
such that $q(x,y)=0$.
By Harnack's curve theorem, the set $Q$ has finitely many connected components.
B\'ezout's theorem implies that
the number of branching points in each of these components is finite and these points have finite degrees.
Consequently, $\RR^2\setminus Q$ has finitely many components.
If $A_1,\ldots,A_\ell$ are all the connected components of $[0,1]^2\setminus Q$,
then each function $f\in\FF$ coincides with one of the $k$ polynomials $p_1,\ldots,p_k$ on every $A_i$.
So, $|\FF|\le k^\ell$.

Observe that the function $F_\mu(x,y)$ is continuous since the measure $\mu$ has uniform marginals.
By the Stone-Weierstrass theorem, there exist a polynomial $p(x,y)$ and $\varepsilon>0$ such that
\begin{eqnarray}
\int_{[0,1]^2}\left( F_\mu(x,y)-p(x,y) \right)^2\dif\lambda & < & \varepsilon \;\mbox{, and} \label{eq-force1} \\
\int_{[0,1]^2}\left( f(x,y)-p(x,y) \right)^2\dif\lambda & > & \varepsilon \mbox{ for every $f\in\FF$, $f\not=F_\mu$.} \label{eq-force2}
\end{eqnarray}
Let $\varepsilon_0$ be the value of the left hand side of (\ref{eq-force1}).
We claim that the unique permuton $\mu'$ satisfying
\begin{eqnarray}
\int_{[0,1]^2}\prod_{i=1}^k\left( F_{\mu'}(x,y)-p_i(x,y) \right)^2\dif\lambda & = & 0 \;\mbox{, and} \label{eq-force3} \\
\int_{[0,1]^2}\left( F_{\mu'}(x,y)-p(x,y) \right)^2\dif\lambda & = & \varepsilon_0 \;  \label{eq-force4}
\end{eqnarray}
is $\mu$.
Assume that $\mu'$ is a permuton satisfying both (\ref{eq-force3}) and (\ref{eq-force4}).
The equation (\ref{eq-force3}) implies that $F_{\mu'}\in\FF$.
Next,
(\ref{eq-force2}), (\ref{eq-force4}), and \eqref{eq-force1} yield that $F_{\mu'}\not=f$ for every $f\in\FF$, $f\not=F_\mu$.
We conclude that $F_{\mu'}=F_{\mu}$ and thus $\mu'=\mu$.

By Theorem~\ref{thm-express}, the left hand sides of (\ref{eq-force3}) and (\ref{eq-force4})
can be expressed as finite linear combinations of densities $d(\sigma,\mu)$.
Let $S$ be the set of all permutations appearing in these linear combinations.
Any permuton $\mu'$ with $d(\sigma,\mu')=d(\sigma,\mu)$ for every $\sigma\in S$
satisfies both (\ref{eq-force3}) and (\ref{eq-force4}) and thus it must be equal to $\mu$.
This shows that $\mu$ is finitely forcible.
\end{proof}

We immediately obtain the following corollary.

\begin{corollary}
\label{cor-force}
If $\mu$ is a permuton such that
there exist non-negative reals $\alpha_1,\ldots,\alpha_k$ and
non-trivial polygons $A_1,\ldots,A_k\subseteq [0,1]^2$
satisfying $\mu=\sum\limits_{i=1}^k\alpha_i\lambda_{A_i}$,
then $\mu$ is finitely forcible.
\end{corollary}

\begin{proof}
Let $F_i$, $i\in [k]$, be the function from $[0,1]^2$ to $[0,1]$ defined
as $F_i(x,y)=\lambda_{A_i}\lt([0,x]\times [0,y]\rt)$.
Clearly, each function $F_i$ is piecewise polynomial.
Since $F_{\mu}=\sum\limits_{i=1}^k\alpha_i F_i$,
the finite forcibility of $\mu$ follows from Theorem~\ref{thm-force}.
\end{proof}

\begin{figure}
\begin{center}
\epsfbox{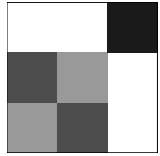}
\end{center}
\caption{The permuton $\mu_{M}$ constructed as an example at the end of Section~\ref{sect-finite}.
The gray area in the picture is the support of the measure and different shades correspond to the density of the measure.}\label{fig-force-step}
\end{figure}

A particular case of permutons that are finitely forcible by Corollary~\ref{cor-force} is the following.
If $k$ is an integer, $z_1,\ldots,z_k\in [0,1]$ are reals such that $z_1+\cdots+z_k=1$ and
$M$ is a square matrix of order $k$ with entries being non-negative reals summing to $z_i$
in the $i$-th row and in the $i$-th column, we can define a permuton $\mu_{M}$ to be the sum
\[\mu_{M}=\sum_{i,j=1}^k M_{ij}\mu_{A_{ij}}\;\mbox{,}\]
where $A_{ij}=[s_{i-1} ,s_i]\times [s_{j-1},s_j]$, $i,j\in [k]$ and
$s_i=z_1+\cdots+z_i$ (in particular, $s_0=0$ and $s_k=1$).
For instance, if $z_1=z_2=z_3=1/3$ and
\[M=\left(\begin{array}{ccc} 0 & 0 & 1/3 \\ 2/9 & 1/9 & 0 \\ 1/9 & 2/9 & 0 \end{array}\right)\;\mbox{,}\]
we get the permuton depicted in Figure~\ref{fig-force-step}.

\section{Permutons with infinite structure}
\label{sect-infinite}

In this section, we show that two particular types of permutons with infinite structure are finitely forcible.
We show that the associated graphons are not finitely forcible in the next section.

\subsection{Union of monotone permutations}
\label{subs-mono}

For $\alpha\in (0,1)$, define $\mu^m_\alpha$ to be the permuton
$\mu^m_\alpha=\sum_{i=1}^\infty (1-\a)\alpha^{i-1}\lambda_{I\lt(1-\alpha^{i-1},1-\alpha^i\rt)}\;\mbox{,}$
where $I(z,z')=\lt\{(x,z'+z-x), x\in [z,z']\rt\}$.
Examples of the just defined permutons can be found in Figure~\ref{fig-mono}.
We next show that all permutons $\mu^m_\alpha$ are finitely forcible.

\begin{figure}
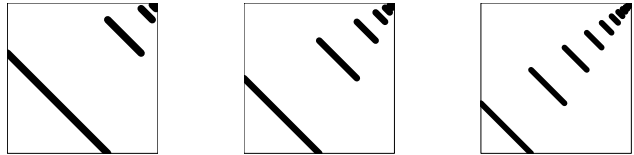

\begin{center}
\epsfbox{permforce.6} \hskip 10mm
\epsfbox{permforce.7} \hskip 10mm
\epsfbox{permforce.8}
\end{center}
\caption{The permutons $\mu^m_{1/3}$, $\mu^m_{1/2}$, and $\mu^m_{2/3}$.}\label{fig-mono}
\end{figure}

\begin{theorem}
\label{thm-mono}
For every $\alpha\in (0,1)$, the permuton $\mu^m_\alpha$ is finitely forcible.
\end{theorem}

\begin{proof}
We claim that any permuton $\mu$ satisfying
\begin{eqnarray}
& d(231,\mu)+d(312,\mu) = 0 \;\mbox{,} & \label{eq-mono1} \\
& d(21,\mu) = (1-\alpha)^2\sum\limits_{i=0}^\infty \alpha^{2i} \;\mbox{, and} & \label{eq-mono1b} \\
& \int\limits_{[0,1]^2}\left(1-x-y+F_\mu(x,y)-\frac{\alpha}{1-\alpha}\lt(x+y-2F_\mu(x,y)\rt)\right)^2\dif\mu = 0 & \label{eq-mono2}
\end{eqnarray}
is equal to $\mu^m_{\alpha}$.
This would prove the finite forcibility of $\mu^m_{\alpha}$ by Theorem~\ref{thm-express-mu}.
Note that the permuton $\mu^m_\alpha$ satisfies (\ref{eq-mono1}), (\ref{eq-mono1b}), and (\ref{eq-mono2}).

Assume that a permuton $\mu$ satisfies (\ref{eq-mono1}), (\ref{eq-mono1b}), and (\ref{eq-mono2}).
Let $X$ be the support of $\mu$ and consider the binary relation $R$ defined on the support of $\mu$ such that
$(x,y)R(x',y')$ if
\begin{itemize}
\item $x=x'$ and $y=y'$, or
\item $x<x'$ and $y>y'$, or
\item $x>x'$ and $y<y'$.
\end{itemize}
The relation $R$ is an equivalence relation. Indeed, the reflexivity and symmetry is clear.
To prove transitivity, consider three points $(x,y)$, $(x',y')$ and $(x'',y'')$ such that
$(x,y)R(x',y')$ and $(x',y')R(x'',y'')$ but it does not hold that $(x,y)R(x'',y'')$.
By the definition of $R$, either $x<x'$ and $x''<x'$, or $x>x'$ and $x''>x'$.
If $x<x'$ and $x''<x'$, then we obtain that $d(231,\mu)>0$ unless $x=x''$ (recall that $R$ is defined on the support of $\mu$).
We can now assume that $x=x''$ and $y<y''$.
Since $\mu$ has uniform marginals, the support of $\mu$ intersects at least one of the open rectangles
$(0,x)\times(y,y'')$, $(x,x')\times(y,y'')$ and $(x',1)\times(y,y'')$. However,
this yields that $d(231,\mu)>0$ in the first two cases and $d(312,\mu)>0$ in the last case.
The case $x>x'$ and $x''>x'$ is handled in an analogous way.

Let $\RRR$ be the set of equivalence classes of $R$.
If $A\in\RRR$, let $A_x$ and $A_y$ be the projections of $A$ on the $x$ and $y$ axes.
It is not hard to show that $A_x$ is a closed interval for each $A\in\RRR$ and
these intervals are internally disjoint for different choices of $A\in\RRR$.
The same holds for the projections on the $y$ axis.
Since $\mu$ has uniform marginals, the intervals $A_x$ and $A_y$ must have the same length for every $A\in\RRR$.
Moreover, the definition of $R$ implies that if $A_x$ precedes $A'_x$, then $A_y$ also precedes $A'_y$ for any $A,A'\in\RRR$.
We conclude that there exists a set $\II$ of internally disjoint closed intervals such that
\[\bigcup_{\lt[z,z'\rt]\in\II} \lt[z,z'\rt]=[0,1]\mbox{ and}\]
the support of $\mu$ is equal to (because the density of subpermutations $231$ and $312$ is zero and the measure $\mu$ has uniform marginals)
\[\bigcup_{\lt[z,z'\rt]\in\II} \lt\{(x,z'-x+z),x\in \lt[z,z'\rt]\rt\}\;\mbox{.}\]
Note that some intervals contained in $\II$ may be formed by single points.
Let $\II_0$ be the subset of $\II$ containing the intervals of positive length.

Let $[z,z']\in\II_0$ and let $I=\lt\{(x,z'-x+z),x\in [z,z']\rt\}$.
Since $\mu([0,x]\times [0,y])=\mu([0,z]\times [0,z])$ and the measure $\mu$ has uniform marginals,
it follows that $F_\mu(x,y)=z$.
The equality (\ref{eq-mono2}) implies that the (continuous) function integrated in (\ref{eq-mono2}) is zero for every $(x,y)\in I$.
Substituting $x+y=z+z'$ and $F_\mu(x,y)=z$ into this function implies
\begin{equation}
z'=z+(1-\alpha)(1-z)\;\mbox{.}\label{eq-zz}
\end{equation}

Let $Z$ be the set formed by the left end points of intervals in $\II_0$.
Define $z_1$ to be the minimum elements of $Z$,
and in general $z_i$ to be the minimum element of $Z\setminus\bigcup\limits_{j<i}\{z_j\}$.
The existence of these elements follows from (\ref{eq-zz}) and the fact that the intervals in $\II_0$ are internally disjoint.
If $Z$ is finite, we set $z_k=1$ for $k>|Z|$.
We derive from the definition of $Z$ and from (\ref{eq-zz}) that
\begin{eqnarray*}
\II_0 & = & \lt\{[z_i,z_i+(1-\alpha)(1-z_i)], ~ i\in \mathbb{N}^+\rt\}\setminus\lt\{[1,1]\rt\}\;\mbox{.}
\end{eqnarray*}
Consequently, we obtain
\begin{equation}
d(21,\mu)=\sum_{i=1}^\infty (1-\alpha)^2 (1-z_i)^2=(1-\alpha)^2\sum_{i=1}^\infty (1-z_i)^2\;\mbox{.}\label{eq-d21}
\end{equation}
For $j\in \mathbb{N}$, we define $\beta_j\in [0,1]$ as follows:
\[\beta_j=\left\{
          \begin{array}{ll}
             1-z_1 & \mbox{for } j=0,\\
	  \frac{1-z_j}{\alpha(1-z_{j-1})} & \mbox{if $z_{j}\neq 1$ and $j>0$, and} \\
	  0 & \mbox{otherwise.}
	  \end{array}
          \right.\]
The equation (\ref{eq-d21}) can now be rewritten as
\begin{equation}
d(21,\mu)=(1-\alpha)^2\sum_{i=1}^\infty\alpha^{2(i-1)}\prod_{j=1}^i\beta_j^2\;\mbox{.}
\end{equation}
Hence, the equality (\ref{eq-mono1b}) can hold only if $\beta_j=1$ for every $j$
which implies that $z_i=1-\alpha^{i-1}$. Consequently, the permutons $\mu$ and $\mu^m_\alpha$
are identical.
\end{proof}

We remark that any permuton $\mu$ obeying the constraints (\ref{eq-mono1}) and (\ref{eq-mono2})
must also be equal to $\mu^m_\alpha$.
However, we decided to include the additional constraint (\ref{eq-mono1b})
to make the presented arguments more straightforward.

\subsection{Union of random permutations}
\label{subs-rand}

We now show finite forcibility of another class of permutons. Its structure is similar to that of $\mu^m_\alpha$.
The proof proceeds along similar lines as the proof of Theorem~\ref{thm-mono}
but we have to overcome several new technical difficulties.
For $\alpha\in (0,1)$, define $\mu^r_\alpha$ to be the permuton
$\mu^r_\alpha=\sum_{i=1}^\infty (1-\a)\alpha^{i-1}\lambda_{[1-\alpha^{i-1},1-\alpha^i]\times[1-\alpha^{i-1},1-\alpha^i]}\;\mbox{.}$
See Figure~\ref{fig-rand} for examples.
Our goal is to show that all permutons $\mu^r_\alpha$ are finitely forcible.

\begin{figure}
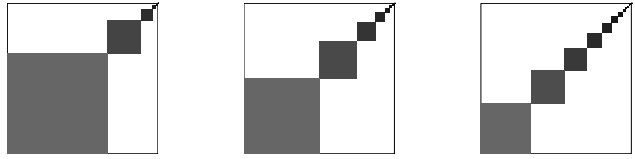

\begin{center}
\epsfbox{permforce.9} \hskip 10mm
\epsfbox{permforce.10} \hskip 10mm
\epsfbox{permforce.11}
\end{center}
\caption{The permutons $\mu^r_{1/3}$, $\mu^r_{1/2}$, and $\mu^r_{2/3}$.}\label{fig-rand}
\end{figure}

We start by proving an auxiliary lemma.

\begin{lemma}
\label{lm-multiint2}
There exist a finite set $S$ of permutations and reals $\gamma_\sigma$, $\sigma\in S$, such that
the following is equivalent for every permuton $\mu$:
\begin{itemize}
\item $\sum\limits_{\sigma\in S}\gamma_\sigma d(\sigma,\mu)=0$,
\item $\mu$ restricted to $[x_1,x_2]\times[y_2,y_1]$ is a (possibly zero) multiple of $\lambda_{[x_1,x_2]\times[y_2,y_1]}$
      for any two points $(x_1,y_1)$ and $(x_2,y_2)$ of the support of $\mu$ with $x_1<x_2$ and $y_1>y_2$.
\end{itemize}
\end{lemma}

\begin{proof}
The proof technique is similar to that used in~\cite{bib-kral12+}.
Let $\tilde{\lambda}_A(X)=\lambda(X\cap A)$, i.e., $\tilde{\lambda}_A(X)=\lambda(A) \cdot \lambda_A(X)$.
Let $(x_1,y_1)$ and $(x_2,y_2)$ be two points of the support of $\mu$ with $x_1<x_2$ and $y_1>y_2$.
By Cauchy-Schwartz inequality,
the measure $\mu$ restricted $[x_1,x_2]\times[y_2,y_1]$ is a multiple of $\tilde{\lambda}_{[x_1,x_2]\times[y_2,y_1]}$
if and only if it holds that
\begin{align}
&\left(\int\limits_{(x,y)}\mu([x_1,x]\times [y_2,y])^2\;\dif\tilde{\lambda}_{[x_1,x_2]\times[y_2,y_1]}\right)\times  \nonumber \\
&\left(\int\limits_{(x,y)}(x-x_1)^2(y-y_2)^2\;\dif\tilde{\lambda}_{[x_1,x_2]\times[y_2,y_1]}\right)-  \nonumber \\
&\left(\int\limits_{(x,y)}(x-x_1)(y-y_2)\mu([x_1,x]\times [y_2,y])\;\dif\tilde{\lambda}_{[x_1,x_2]\times[y_2,y_1]}\right)^2  =  0 \label{eq-multi-rand1}
\end{align}
Since the left hand side of (\ref{eq-multi-rand1})
cannot be negative, we obtain that the second statement in the lemma is equivalent to
\begin{align}
&\int\limits_{(x_1,y_1)}\int\limits_{(x_2,y_2)}\int\limits_{(x,y)}\int\limits_{(x',y')}
(x'-x_1)^2(y'-y_2)^2\cdot\mu\lt([x_1,x]\times [y_2,y]\rt)^2- & & \nonumber \\
&(x-x_1)(y-y_2)\cdot\mu\lt([x_1,x]\times [y_2,y]\rt)\cdot(x'-x_1)(y'-y_2)\cdot\mu([x_1,y_2]\times [x',y']) & & \nonumber \\
&\qquad\qquad\qquad\qquad\qquad\qquad
\;\dif\tilde{\lambda}_{[x_1,x_2]\times[y_2,y_1]}\;\dif\tilde{\lambda}_{[x_1,x_2]\times[y_2,y_1]}\;\dif\mu\;\dif\mu ~=~ 0  \label{eq-multi-rand2}
\end{align}
In the rest of the proof,
we show that the left hand side of (\ref{eq-multi-rand2}) can be expressed as a linear combination of finitely many subpermutation  densities.
Since this argument follows the lines of the proofs of Theorems~\ref{thm-express}--\ref{thm-express-mu-general},
we only briefly explain the main steps.

The left hand side of (\ref{eq-multi-rand2}) is equal to the expected value of the integrated function in (\ref{eq-multi-rand2})
for two points $(x_1,y_1)$ and $(x_2,y_2)$ randomly chosen in $[0,1]^2$ based on $\mu$ and
two points $(x,y)$ and $(x',y')$ randomly chosen in $[0,1]^2$ based on $\lambda$
when treating the value of the integrated function to be zero
if $x_1\ge x_2$, $y_1\ge y_2$, $x\not\in [x_1,x_2]$, $x'\not\in [x_1,x_2]$, $y\not\in [y_1,y_2]$, or $y'\not\in [y_1,y_2]$.
Such points $(x_1,y_1)$, $(x_2,y_2)$, $(x,y)$, and $(x',y')$ can be obtained by sampling six random points from $[0,1]^2$ based on $\mu$
since $\mu$ has uniform marginals (see the proof of Theorem~\ref{thm-express} for more details).
When the four points $(x_1,y_1)$, $(x_2,y_2)$, $(x,y)$, and $(x',y')$ are sampled,
any of the quantities $x_1$, $y_2$, $x$, $y$, $x'$, $y'$, $\mu\lt([x_1,y_2]\times [x,y]\rt)$, and $\mu([x_1,y_2]\times [x',y'])$
appearing in the product is equal to the probability that a point randomly chosen in $[0,1]^2$ based on $\mu$
has a certain property in a permutation determined by the sampled points.
Since we need to sample six additional points to be able to determine each of the products appearing in (\ref{eq-multi-rand1}),
the left hand side of (\ref{eq-multi-rand1}) is equal to a linear combination of densities of $12$-element permutations
with appropriate coefficients. We conclude that the lemma holds with $S=S_{12}$.
\end{proof}

Analogously, one can prove the following lemma. Since the proof follows the lines
of the proof of Lemma~\ref{lm-multiint2}, we omit further details.

\begin{lemma}
\label{lm-multiint3}
There exist a finite set $S$ of permutations and reals $\gamma_\sigma$, $\sigma\in S$ such that
the following is equivalent for every permuton $\mu$:
\begin{itemize}
\item $\sum_{\sigma\in S}\gamma_\sigma d(\sigma,\mu)=0$,
\item if $(x_1,y_1)$, $(x_2,y_2)$, and $(x_3,y_3)$ are three points of the support of $\mu$ with $x_1<x_2<x_3$ and $y_2<y_3<y_1$,
      then $\mu$ restricted $[x_2,x_3]\times[y_2,y_3]$ is a (possibly zero) multiple of $\lambda_{[x_2,x_3]\times[y_2,y_3]}$.
\end{itemize}
\end{lemma}

We are now ready to show that each permuton $\mu^r_\alpha$, $\alpha\in (0,1)$, is finitely forcible.

\begin{theorem}
\label{thm-rand}
For every $\alpha\in (0,1)$, the permuton $\mu^r_\alpha$ is finitely forcible.
\end{theorem}

\begin{proof}
Let $S_0$ be the union of the two sets of permutations from Lemmas~\ref{lm-multiint2}~and~\ref{lm-multiint3}.
Next, consider the following eight functions:
\[
\begin{array}{cclcccl}
F_\mu^\nwarrow(x,y) & = & F_\mu^{2\overline{1}}(x,y)\;\mbox{,} & \hskip 1cm &
f_\mu^\nwarrow(x,y) & = & F_\mu^{23\overline{1}}(x,y)+F_\mu^{32\overline{1}}(x,y)\;\mbox{,} \\
F_\mu^\nearrow(x,y) & = & F_\mu^{\overline{1}2}(x,y)\;\mbox{,} & \hskip 1cm &
f_\mu^\nearrow(x,y) & = & F_\mu^{\overline{2}31}(x,y)\;\mbox{,} \\
F_\mu^\swarrow(x,y) & = & F_\mu^{1\overline{2}}(x,y)\;\mbox{,} & \hskip 1cm &
f_\mu^\swarrow(x,y) & = & F_\mu^{31\overline{2}}(x,y)\;\mbox{,} \\
F_\mu^\searrow(x,y) & = & F_\mu^{\overline{2}1}(x,y)\;\mbox{,} & \hskip 1cm &
f_\mu^\searrow(x,y) & = & F_\mu^{\overline{3}12}(x,y)+F_\mu^{\overline{3}21}(x,y)\;\mbox{.}
\end{array}
\]
To save space in what follows, we often omit parameters when no confusion can arise, e.g.,
we write $F_\mu^\searrow$ for the value $F_\mu^\searrow(x,y)$ if $x$ and $y$ are clear from the context.

We claim that any permuton satisfying the following three conditions is equal to $\mu^r_{\alpha}$:
\begin{equation}
d(\sigma,\mu) = d(\sigma,\mu^r_{\alpha}) \mbox{ for every $\sigma\in S_0$,}  \label{eq-rand1}
\end{equation}
\begin{align}
&\int\limits_{[0,1]^2}\Big((1-\alpha)\lt(F_\mu^\nearrow f_\mu^\searrow-F_\mu^\searrow f_\mu^\nearrow\rt)f_\mu^\nwarrow \nonumber\\
&\qquad              -
                    \alpha\big(F_\mu^\nwarrow f_\mu^\nwarrow f_\mu^\searrow+        F_\mu^\searrow f_\mu^\nwarrow f_\mu^\searrow+
                           F_\mu^\nwarrow f_\mu^\swarrow f_\mu^\searrow+
                           F_\mu^\searrow f_\mu^\nwarrow f_\mu^\nearrow\big)
		    \Big)^2\;\dif\mu = 0 \; , \label{eq-rand1b}
\end{align}
and
\begin{equation}
d(21,\mu) = \frac{(1-\alpha)^2}{2}\sum\limits_{i=0}^\infty \alpha^{2i} \;\mbox{.}  \label{eq-rand2}
\end{equation}
This would prove the finite forcibility of $\mu^r_{\alpha}$ by Theorem~\ref{thm-express-mu-general}.

Suppose that a permuton $\mu$ satisfies (\ref{eq-rand1}),  (\ref{eq-rand1b}), and (\ref{eq-rand2}).
Let $X$ be the support of $\mu$ and consider the binary relation $R$ defined on the support of $\mu$ such that
$(x,y)R(x',y')$ if
\begin{itemize}
\item $x=x'$ and $y=y'$,
\item $x<x'$ and $y>y'$, or
\item $x>x'$ and $y<y'$.
\end{itemize}
Unlike in the proof of Theorem~\ref{thm-mono},
the relation $R$ need not be an equivalence relation.
Instead, we consider the transitive closure $R_0$ of $R$ and
let $\RRR$ be the set of the equivalence classes of $R_0$.

We define $\rho((x,y),(x',y'))$,
where $(x,y)$ and $(x',y')$ are two points of the support of $\mu$ such that $(x,y)R(x',y')$,
as follows
\[
\rho\lt((x,y),(x',y')\rt)=\left\{
  \begin{array}{cl}
  \frac{\mu\lt([x,x']\times [y',y]\rt)}{(x'-x)(y-y')} & \mbox {if $x<x'$ and $y>y'$,} \\
  \frac{\mu\lt([x',x]\times [y,y']\rt)}{(x-x')(y'-y)} & \mbox {if $x>x'$ and $y<y'$, and} \\
  0 & \mbox{otherwise.}
  \end{array} \right. \]
Since $\mu$ satisfies (\ref{eq-rand1}),
Lemma~\ref{lm-multiint2} implies that
any three points $(x,y)$, $(x',y')$ and $(x'',y'')$ of the support of $\mu$ such that $(x,y)R(x',y')$ and $(x',y')R(x'',y'')$
satisfy $\rho((x,y),(x',y'))=\rho((x',y'),(x'',y''))$.
In particular, the quantity $\rho((x,y),(x',y'))$
is the same for all pairs of points $(x,y)$ and $(x',y')$ with $(x,y)R(x',y')$ lying in the same equivalence class of $R_0$.
So, we may define $\rho(A)$ to be this common value for each equivalence class $A\in\RRR$
or for a closure of such class.

As in the proof of Theorem~\ref{thm-mono},
we define $A_x$ and $A_y$ to be the projections of an equivalence class $A\in\RRR$ on the $x$ and $y$ axes.
The definition of $R$ yields that $A_x$ and $A_y$ are closed intervals for all $A\in\RRR$ and
these intervals are internally disjoint for different choices of $A\in\RRR$.
Since $\mu$ has uniform marginals, the intervals $A_x$ and $A_y$ must have the same length for every $A\in\RRR$.
As in the proof of Theorem~\ref{thm-mono},
we conclude that there exists a set $\II$ of internally disjoint closed intervals such that
\[\bigcup_{[z,z']\in\II} [z,z']=[0,1]\; \mbox{,}\]
the support of $\mu$ is a subset of
\[\bigcup_{[z,z']\in\II} [z,z']\times [z,z']\;\mbox{,}\]
and the interior of each of these squares intersects at most one class $A\in \RRR$.
Since some intervals contained in $\II$ may be formed by single points,
we define $\II_0$ to be the subset of $\II$ containing the intervals of positive length.

Let $[z,z']\in\II_0$ and let $A$ be the closure of the corresponding equivalence class from $\RRR$.
Let $f(x)$, $x\in [z,z']$, be the minimum $y$ such that $(x,y)$ belongs to $A$;
similarly, $g(x)$ denotes the maximum such $y$.

Assume first that $\rho(A)>0$. Since $\mu$ has uniform marginals,
it must holds that $g(x)-f(x)=\rho(A)^{-1}$ for every $x\in (z,z')$.
From \eqref{eq-rand1} and Lemma~\ref{lm-multiint2} we see that the functions $f$ and $g$ are non-decreasing,
and similarly \eqref{eq-rand1} and Lemma~\ref{lm-multiint3} imply that $f$ and $g$ are non-increasing.
We conclude that $A=\lt([z,z']\times [z,z']\rt)$ and
$\rho(A)=(z'-z)^{-1}$.

Assume now that $\rho(A)=0$.
Lemma~\ref{lm-multiint3} and (\ref{eq-rand1}) imply that
if $(x,y)\in (z,z')\times (z,z')$ belongs to the support of $\mu$,
then $\mu([x,z']\times [z,y])=0$ (otherwise, $\rho(A)>0$).
But then $(x,y)$ cannot be in relation $R$ with another point of the support of $\mu$.
So, we conclude that the case $\rho(A)=0$ cannot appear.

The just presented arguments show the support of the measure $\mu$
is equal to
\[\bigcup_{[z,z']\in\II} [z,z']\times [z,z']\]
and the measure is uniformly distributed inside each square $[z,z']\times [z,z']$, $[z,z']\in\II_0$.

Let $[z,z']$ be one of the intervals from $\II_0$.
Recall that we have argued that
\[\mu\big([0,z]\times[0,z]\cup [z,z']\times [z,z']\cup [z',1]\times [z',1]\big)=1\]
and
the measure $\mu$ is uniform inside the square $[z,z']\times [z,z']$ (see Figure~\ref{fig-rand-zz}).
By (\ref{eq-rand1b}), the following holds for almost every point of the support of $\mu$:
\begin{eqnarray}
(1-\alpha)\lt(F_\mu^\nearrow f_\mu^\searrow-F_\mu^\searrow f_\mu^\nearrow\rt)f_\mu^\nwarrow &=&
                   \alpha\big(F_\mu^\nwarrow f_\mu^\nwarrow f_\mu^\searrow+
                          F_\mu^\searrow f_\mu^\nwarrow f_\mu^\searrow+\nonumber\\
& &                       F_\mu^\nwarrow f_\mu^\swarrow f_\mu^\searrow+
                          F_\mu^\searrow f_\mu^\nwarrow f_\mu^\nearrow\big)\;\mbox{.}
\label{eq-rand1c}
\end{eqnarray}
In particular, this holds for all points in $[z,z']\times [z,z']$
since the functions appearing in (\ref{eq-rand1c}) are continuous.

\begin{figure}
\begin{center}
\epsfbox{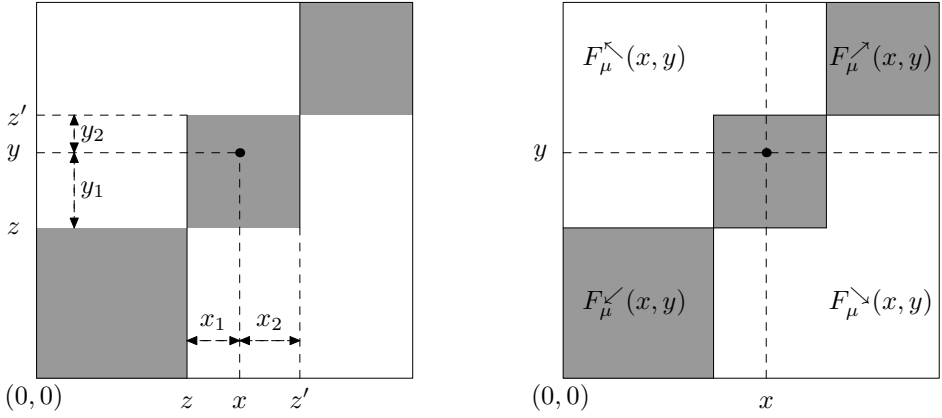}
\end{center}
\caption{Notation used in equalities (\ref{eq-rand1d}) and (\ref{eq-rand1f}).
         Areas that can contain the support of $\mu$ are drawn in grey.}
\label{fig-rand-zz}
\end{figure}

Let $(x,y)$ be a point from $(z,z')\times (z,z')$. Let $x_1=x-z$, $x_2=z'-x$, $y_1=y-z$, and $y_2=z'-y$ (see Figure~\ref{fig-rand-zz}).
Since all the quantities appearing in (\ref{eq-rand1c}) are positive,
we may rewrite (\ref{eq-rand1c}) as
\begin{equation}
(1-\alpha)\left(F_\mu^\nearrow-F_\mu^\searrow\frac{f_\mu^\nearrow}{f_\mu^\searrow}\right)=
   \alpha \left(F_\mu^\nwarrow+F_\mu^\searrow+F_\mu^\nwarrow\frac{f_\mu^\swarrow}{f_\mu^\nwarrow}+F_\mu^\searrow\frac{f_\mu^\nearrow}{f_\mu^\searrow}\right)\;\mbox{.}
\label{eq-rand1d}
\end{equation}
Observe that $F_\mu^\nearrow(x,y)=\mu([x,1]\times [y,1])$,
$F_\mu^\nwarrow(x,y)=\mu([z,x]\times [y,z'])=\frac{x_1y_2}{z'-z}$, and
$F_\mu^\searrow(x,y)=\mu([x,z']\times [z,y])=\frac{x_2y_1}{z'-z}$.
Further observe that
\[
\frac{f_\mu^\nearrow(x,y)}{f_\mu^\searrow(x,y)}=\frac{\frac{2x_2^2y_1y_2}{2(z'-z)^2}}{\frac{x_2^2y_1^2}{(z'-z)^2}}=\frac{y_2}{y_1}~~\mbox{  and  }~~
\frac{f_\mu^\swarrow(x,y)}{f_\mu^\nwarrow(x,y)}=\frac{\frac{2x_1^2y_1y_2}{2(z'-z)^2}}{\frac{x_1^2y_2^2}{(z'-z)^2}}=\frac{y_1}{y_2}\;\mbox{.}
\]
Plugging these observations in (\ref{eq-rand1d}), we obtain that
\begin{equation}
(1-\alpha)\left(\mu([x,1]\times [y,1])-\frac{x_2y_2}{z'-z}\right)=\alpha\frac{x_1y_2+x_2y_1+x_1y_1+x_2y_2}{z'-z}\;\mbox{.}
\label{eq-rand1f}
\end{equation}
Since $x_1+x_2=y_1+y_2=z'-z$ and $\frac{x_2y_2}{z'-z}=\mu([x,z']\times[y,z'])$,
we obtain from (\ref{eq-rand1f}) that
\begin{equation}
(1-\alpha)\mu([z',1]\times [z',1])=\alpha\frac{(z'-z)^2}{z'-z}=\a(z'-z)\;\mbox{.}
\label{eq-rand1g}
\end{equation}
Finally, we substitute $1-z'$ for $\mu([z',1]\times[z',1])$ in (\ref{eq-rand1g}) and get the following:
\begin{equation}
z'=z+(1-\alpha)(1-z)\;\mbox{.}\label{eq-rand-zz}
\end{equation}
So, we conclude that the right end point of every interval in $\II_0$ is uniquely determined by its left end point.

Let $Z$ be the set formed by the left end points of intervals in $\II_0$.
As in the proof of Theorem~\ref{thm-mono}, for a positive integer $i$,
let $z_i$ be the $i$th smallest element of $Z$.
Notice that the existence of minimum elements follows from (\ref{eq-rand-zz}).
If $Z$ is finite, we set $z_k=1$ for $k>|Z|$.

We derive from the definition of $Z$ and from (\ref{eq-rand-zz}) that
\begin{eqnarray*}
\II_0 & = & \lt\{[z_i,z_i+(1-\alpha)(1-z_i)], ~ i\in \mathbb{N}^+\rt\}\setminus\lt\{[1,1]\rt\}\;\mbox{.}
\end{eqnarray*}
Consequently, we obtain
\begin{equation}
d(21,\mu)=\sum_{i=1}^\infty \frac{(1-\alpha)^2 (1-z_i)^2}{2}=\frac{(1-\alpha)^2}{2}\sum_{i=1}^\infty (1-z_i)^2\;\mbox{.}\label{eq-rand-d21}
\end{equation}
Analogously to the proof of Theorem~\ref{thm-mono},
for $j\in \mathbb{N}$, we define $\beta_j\in [0,1]$ as follows:
\[\beta_j=\left\{
          \begin{array}{ll}
             1-z_1 & \mbox{for } j=0,\\
	  \frac{1-z_j}{\alpha(1-z_{j-1})} & \mbox{if $z_{j}\neq 1$ and $j>0$, and} \\
	  0 & \mbox{otherwise.}
	  \end{array}
          \right.\]
The equation (\ref{eq-rand-d21}) can be rewritten as
\begin{equation}
d(21,\mu)=\frac{(1-\alpha)^2}{2}\sum_{i=1}^\infty\alpha^{2(i-1)}\prod_{j=1}^i\beta_j^2;\mbox{.}
\end{equation}
Hence, the equality (\ref{eq-rand2}) can hold only if $\beta_j=1$ for every $j$,
i.e., $z_i=1-\alpha^{i-1}$.
This implies that the permutons $\mu$ and $\mu^r_\alpha$ are identical.
\end{proof}

\section{Graphons with infinite structure}
\label{sect-graphon}

In this section, we show that two types of graphons similar to the two types of finitely forcible permutons
from Section~\ref{sect-infinite} are not finitely forcible.
We start with graphons $W_{\mu^m_\alpha}$ associated with permutons $\mu^m_\alpha$, $\alpha\in(0,1)$.

\subsection{Union of complete graphs}
\label{sect-clique}

In this subsection we focus on graphons with $d(P_3,W)=0$ where $P_3$ is the path on $3$ vertices.
We start with the following lemma, which seems to be of be of independent interest.
Informally, the lemma asserts that any finitely forcible graphon with zero density of $P_3$ can be forced
by finitely many densities of complete graphs.

\begin{lemma}
\label{lm-clique}
If $W_0$ is a finitely forcible graphon and $d(P_3,W_0)=0$,
then there exists an integer $\ell_0$ such that
any graphon $W$ with $d(P_3,W)=0$ and $d(K_\ell,W)=d(K_\ell,W_0)$ for $\ell\leq \ell_0$ is weakly isomorphic to $W_0$.
\end{lemma}

\begin{proof}
To prove the statement of the lemma,
it is enough to show the following claim:
the density of any $n$-vertex graph $G$ in a graphon $W$ with $d(P_3,W)=0$
can be expressed as a combination of
densities of $K_1,\ldots$, $K_n$ in $W$.
We proceed by induction on $n+k$ where $n$ and $k$ are the numbers of vertices and components of $G$ respectively.
If $n=k=1$, there exists only a single one-vertex graph $K_1$ and the claim holds.

Assume now that $n+k>2$.
If $G$ is not a union of complete graphs, then $d(G,W)=0$ since $d(P_3,W)=0$.
So, we assume that $G$ is a union of $k$ complete graphs $G_1,\ldots,G_k$,
i.e., $G=G_1\cup\cdots\cup G_k$.
If $k=1$, then $G=K_n$ and the claim clearly holds. So, we assume $k>1$.

For $2\leq i\leq k$, we denote
\[ H_i = (G_1+G_i)\cup\bigcup_{j\in[k]\setminus\{1,i\}}G_j\; .\]
Observe that the following holds:
\begin{align}
&d(G_1,W)d(G_2\cup\cdots\cup G_k,W)=\nonumber\\
                                  & \qquad p_1\;\cdot\;d\left(G_1\cup\cdots\cup G_k,W\right)+\sum_{i=2}^k p_i\;\cdot\;d\left(H_i,W\right)
\label{eq-clique-mul}
\end{align}
where $p_1$ is the probability that a set $V$ of randomly chosen $|G_1|$ vertices of the graph $G_1\cup\cdots\cup G_k$
induces a complete graph and the graph $(G_1\cup\cdots\cup G_k)\setminus V$ is isomorphic to $G_2\cup\cdots\cup G_k$, and
$p_i$, $i>1$, is the probability that a set $V$ of randomly chosen $|G_1|$ vertices of $H_i$
induces a complete graph and the graph $H_i\setminus V$ is isomorphic to $G_2\cup\cdots\cup G_k$.
To see (\ref{eq-clique-mul}), observe that
the product $d(G_1,W)d(G_2\cup\cdots\cup G_k,W)$ is equal to the product of
the probability that a $W$-random graph of order $|G_1|$ is isomorphic to $G_1$ and
the probability that a $W$-random graph of order $|G_2|+\cdots+|G_k|$ is isomorphic to $G_2\cup\cdots\cup G_k$.
This is equal to the probability that
randomly chosen $|G_1|$ vertices of a $W$-random graph of order $|G_1|+\cdots+|G_k|$ induce a subgraph isomorphic to $G_1$ and
the remaining vertices induce a subgraph isomorphic to $G_2\cup\cdots\cup G_k$.
This probability is equal to the right hand side of (\ref{eq-clique-mul}).

By induction, $d(G_2\cup\cdots\cup G_k,W)$ and $d\left(H_i,W\right)$, $2\leq i\leq k$,
can be expressed as combinations of densities of complete graphs of order at most $n$ in $W$.
Rearranging the terms of (\ref{eq-clique-mul}), we obtain that $d(G_1\cup\cdots\cup G_k,W)$
is equal to a combination of densities of complete graphs of order at most $n$ in $W$.
\end{proof}

For a sequence of non-negative reals $\vec a=( a_i)_{i\in\NN}$ such that $\sum_{i=1}^\infty a_i=1$,
define a graphon $W^c_{\vec a}$ such that $W(x,y)=1$ if and only if there exists $j\in\NN$ such that
\[\sum_{i=1}^{j-1} a_i\le x,y\le\sum_{i=1}^{j} a_i\;\mbox{.}\]
We consider a particular case of this graphon $W^c_\alpha$ for $\alpha\in (0,1)$:
$W^c_\alpha=W^c_{\vec a}$ for $a_i=(1-\alpha)\alpha^{i-1}$.
Observe that $W^c_\alpha=W_{\mu^m_\alpha}$. The main result of this subsection asserts that
unlike the associated permuton $\mu^m_\alpha$, the graphon $W^c_\alpha$ is not finitely forcible.
Although it immediately follows as a corollary of the more general Theorem~\ref{thm-quasi}, we give its proof here to increase the readability of the paper.

\begin{theorem}
\label{thm-clique}
For every $\alpha\in (0,1)$, the graphon $W^c_\alpha=W_{\mu^m_\alpha}$ is not finitely forcible.
\end{theorem}

\begin{proof}
Observe that $d\lt(P_3, W^c_\alpha\rt)=0$.
By Lemma~\ref{lm-clique},
it is enough to show that $W^c_\alpha$ is not finitely forcible with $S=\{P_3,K_1,\ldots,K_n\}$ for any $n\in \NN$,
i.e., by setting the densities of $P_3$ and the complete graphs of orders $1,\ldots,n$.
Suppose for the sake of contradiction that for some $n\in \NN$ the graphon $W^c_\alpha$ is uniquely determined by the densities of $P_3$ and $K_1,\ldots,K_n$.
Let $a_i=(1-\alpha)\alpha^{i-1}$.
Further, define functions $F_i(x_1,\ldots,x_{n+1}):\RR^{n+1}\to\RR$ for $i=1,\ldots,n$ as follows:
\begin{equation}
F_i(x_1,\ldots,x_{n+1})=\sum_{j=1}^{n+1}\left(x_j^i-a_j^i\right)\;\mbox{.}\label{eq-clique1}
\end{equation}
Observe that if it holds that $x_1+\cdots+x_{n+1}=a_1+\cdots+a_{n+1}$,
which is equivalent to $F_1(x_1,\ldots,x_{n+1})=0$,
then it also holds that
\begin{equation}
d(K_i,W^c_{\vec b})=d(K_i,W^c_\alpha)+F_i(x_1,\ldots,x_{n+1})\mbox{ for $i\in [n]$,}\label{eq-clique2}
\end{equation}
where $\vec b$ is the sequence with $b_i=x_i$ for $i \leq n+1$ and $b_i=a_i$ for $i>n+1$.
Hence, to obtain the desired contradiction, it suffices to prove that there exist functions $g_j(x_{n+1})$, $j\in [n]$, on some open neighborhood of $a_{n+1}$ such that
\begin{equation}
F_i(g_1(x_{n+1}),\ldots,g_n(x_{n+1}),x_{n+1})=0\mbox{ for every $i\in [n]$.}\label{eq-clique3}
\end{equation}
Indeed, if such functions $g_j(x_{n+1})$, $j\in [n]$, exist,
then (\ref{eq-clique2}) yields that
the densities of $K_1,\ldots,K_n$ in the graphon $W^c_{\vec b}$ with $b_i=g_i(x_{n+1})$ for $i\leq n$, $b_{n+1}=x_{n+1}$ and $b_i=a_i$ for $i>n+1$
equal their densities in the graphon $W^c_{\vec a}$.
This implies that $W^c_{\vec a}$ is not forced by the densities of $P_3$ and $K_1,\ldots,K_n$.

We now establish the existence of functions $g_1,\ldots,g_n$ satisfying (\ref{eq-clique3}) on some open neighborhood of $a_{n+1}$.
Observe that
\[\frac{\partial F_i}{\partial x_j}(x_1,\ldots,x_{n+1})=i\cdot x_j^{i-1}\;\mbox{.}\]
We consider the Jacobian matrix of the functions $F_1,\ldots,F_n$ with respect to $x_1,\ldots,x_n$.
The determinant of the Jacobian matrix is equal to
\begin{equation}
n!\prod_{1\le j<j'\le n} \left(x_{j'}-x_j\right)\;\mbox{.}
\end{equation}
Substituting $x_j=a_j$ for $j=1,\ldots,n$, we obtain that the Jacobian matrix has non-zero determinant.
In particular, the Jacobian is non-zero.
The Implicit Function Theorem now implies the existence of the functions $g_1,\ldots,g_n$
satisfying (\ref{eq-clique3}). This concludes the proof.
\end{proof}

\subsection{Union of random graphs}
\label{sect-quasi}

The graphons considered in the previous section were associated with the permutons $\mu^m_\alpha$.
We now consider graphons related to the permutons $\mu^r_\alpha$. In fact, we introduce a more
general concept.
Let $W$ be a graphon.
For a sequence of non-negative reals $\vec a=( a_i)_{i\in\NN}$ such that $\sum_{i=1}^\infty a_i=1$,
define a graphon $W_{\to\vec a}$ as follows.
Informally speaking, we take the graphon $W^c_{\vec{a}}$ and plant a copy of $W$ on each of its ``components''.
Formally, for $x,y\in [0,1)$, let $j_x$ and $j_y$ be the integers
such that
\begin{eqnarray*}
& \sum\limits_{i=1}^{j_x-1} a_i\le x<\sum_{i=1}^{j_x} a_i & \mbox{ and} \\
& \sum\limits_{i=1}^{j_y-1} a_i\le y<\sum_{i=1}^{j_y} a_i\;\mbox{.}
\end{eqnarray*}
If $j_x\not=j_y$, then $W_{\to\vec a}(x,y)=0$.
If $j_x=j_y$, then
\[W_{\to\vec a}(x,y)=W\left(\frac{x-\sum\limits_{i=1}^{j_x-1} a_i}{a_{j_x}},\frac{y-\sum\limits_{i=1}^{j_y-1} a_i}{a_{j_y}}\right)\;\mbox{.}\]
The set of pairs $(x,y)$ with one of the coordinates being equal to zero has measure zero,
so we can for example set the values of $W(x,y)$ for such pairs to be equal to zero.
Similarly as before,
we also define $W_{\to\alpha}$ for $\alpha\in (0,1)$ as $W_{\to\alpha}=W_{\to\vec a}$
where $a_i=(1-\alpha)\alpha^{i-1}$.

If $W^0_\rho$ is the graphon identically equal to $\rho\in [0,1]$,
then $W^0_{1,\to\vec a}$ is $W^c_{\vec a}$ (we put a comma to separate the two indices, the first referring to the density of edges inside connected components,
the second determining the sizes of those). More generally,
we use $W^r_{\rho,\alpha}$ for the graphon $W^0_{\rho,\to\vec a}$.
Examples can be found in Figure~\ref{fig-quasi}.

\begin{figure}
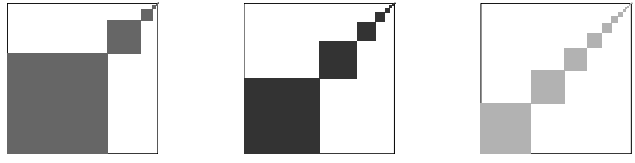

\begin{center}
\epsfbox{permforce.13} \hskip 10mm
\epsfbox{permforce.14} \hskip 10mm
\epsfbox{permforce.15}
\end{center}
\caption{The graphons $W^r_{1/2,1/3}$, $W^r_{3/4,1/2}$, and $W^r_{1/4,2/3}$.}\label{fig-quasi}
\end{figure}

Our main theorem asserts that a graphon $W_{\to\alpha}$ is not finitely forcible unless it is $W^0_0$.

\begin{theorem}
\label{thm-quasi}
For every $\alpha\in (0,1)$ and every graphon $W$,
if the graphon $W_{\to\alpha}$ is finitely forcible,
then $W$ is weakly isomorphic to $W^0_0$, i.e., the graphon $W_{\to\alpha}$ is identically equal to zero up to a set of measure zero.
\end{theorem}

\begin{proof}
It is enough to show that for every $n$, there exists $\vec b$ different from $\vec a$, $a_i=(1-\alpha)\alpha^{i-1}$, such that
\begin{equation}
d\lt(G,W_{\to\vec b}\rt)=d\lt(G,W_{\to\alpha}\rt) \mbox{ for every graph $G$ with $|G|\le n$.}\label{eq-quasi1}
\end{equation}
The proof of Theorem~\ref{thm-clique} yields that for every $n$, there exists such $\vec b$ different from $\vec a$ satisfying
\begin{equation}
d\lt(G,W^c_{\vec b}\rt)=d\lt(G,W^c_{\alpha}\rt)\mbox{ for every graph $G$ with $|G|\le n$.}\label{eq-quasi2}
\end{equation}
We claim that this $\vec b$ also satisfies (\ref{eq-quasi1}).
Also note that (\ref{eq-quasi2}) is non-zero only for graphs $G$ that are disjoint union of cliques.

Let $G$ be a graph with $n$ vertices and let $G_1,\ldots,G_k$ be the connected components of $G$.
Furthermore, let $\FF=\lt\{I_1,\ldots, I_\ell\rt\}$ be the partition of $[k]$ according to the isomorphism classes of the graphs $G_1,\ldots, G_k$, i.e.,
for every $i,j\in [\ell]$ with $i\neq j$
and every $a_1, a_2\in I_i$ and $a_3\in I_j$, the graphs $G_{a_1}$ and $G_{a_2}$ are isomorphic,
and the graphs $G_{a_1}$ and $G_{a_3}$ are not isomorphic.

Observe that
\begin{align}
&d\lt(G,W_{\to\vec b}\rt)=
\sum_{f:[k]\to\NN}
c(f)
\lt( \prod_{i=1}^\infty d\left(\bigcup_{j\in f^{-1}(i)} G_j,W\right)
b_i^{\lt|\bigcup\limits_{j\in f^{-1}(i)} G_j\rt|}\rt)\;
\label{eq-quasi4}
\end{align}
with the normalizing factor
\[
c(f)= \prod\limits_{m\in [\ell]} \frac{\prod\limits_{i=1}^\infty \lt|f ^{-1}(i)\cap I_m\rt|!}{\lt|I_m\rt|!},
\]
where we set $0!=1$ and the density $d(\emptyset,W)$ of the empty graph in the graphon $W$ to $1$.

We consider partitions of the set of connected components of $G$.
If $Q= \lt\{Q_1,\dots, Q_k\rt\}$ is such a partition, we slightly abuse the notation and identify $Q_i$ with the subgraph of $G$ induced by the components of $Q_i$.
In particular, $|Q_i|$ denotes the number of vertices in this subgraph..
Furthermore, we always view a partition $Q$ as a multiset, and also allow some of the $Q_i$'s to be empty.
Let $\cQ$ be the set of all such partitions.
The identity~\eqref{eq-quasi4} can now be rewritten as follows:
\begin{equation}
d\lt(G,W_{\to\vec b}\rt)=
\sum_{Q\in \cQ}\prod\limits_{i\in [k]}d(Q_i, W)
d\lt(\bigcup\limits_{i\in[k]} K_{\lt|Q_i\rt|},W_{\vec b}^c\rt),
\label{eq-quasi5}
\end{equation}
where $K_0$ is the empty graph.
Since $\vec b$ satisfies (\ref{eq-quasi2}), we obtain that it satisfies~\eqref{eq-quasi5}, and therefore also (\ref{eq-quasi1}).
\end{proof}

We immediately obtain the following two corollaries.

\begin{corollary}
\label{cor-quasi}
For every $\alpha\in (0,1)$ and every $\rho\in (0,1]$, the graphon $W^r_{\rho,\alpha}$ is not finitely forcible.
\end{corollary}

\begin{corollary}
\label{cor-grand}
For every $\alpha\in (0,1)$, the graphon $W_{\mu^r_\alpha}=W_{\lambda,\to\alpha}$,
which is associated with the permuton $\mu^r_\alpha$, is not finitely forcible.
\end{corollary}

\section{Conclusion}
\label{sec:concl}

We have shown that graphons associated with finitely forcible permutons need not be
finitely forcible.
Question~11 in~\cite{bib-lovasz11+} is whether there exists
a ``2-dimensional'' finitely forcible graphon.
Such graphons naturally arise from permutons.
In~\cite{bib-comp},
there is a construction of a finitely forcible graphon with the Minkowski dimension of
the associated topological space of typical points being two but the space is not connected.
Our discussions with the authors of~\cite{bib-lovasz11+}
led to the intuition that the graphon $W_{\lambda}$ (as defined at the end of Subsection~\ref{subs-graphon}) is a good candidate
for a finitely forcible graphon with the associated space being connected and having dimension two.

\begin{problem}
Is the graphon $W_\lambda$ associated with the permuton $\lambda$ finitely forcible?
\end{problem}

More generally, we suspect that all graphons associated with permutons $\mu_M$
constructed at the end of Section~\ref{sect-finite} might be finitely forcible.

\begin{problem}
Let $M$ be a square matrix of order $k$ with entries being non-negative reals such
the sum of the entries in the $i$-th row is equal to that in the $i$-column and
the sum of all the entries of $M$ is one.
Is the graphon $W_{\mu_M}$ associated with the permuton $\mu_M$ finitely forcible?
\end{problem}

\section*{Acknowledgements}

This work was done when the second author was visiting the University of Warwick.
The hospitality provided to him by the university is strongly appreciated.
The authors would also like to thank Jan Volec for his comments on the topics related to those covered in the paper.


\begin{thebibliography}{99}
\bibitem{bib-flag1}
R.~Baber:
Tur\'an densities of hypercubes,
preprint available on \url{http://arxiv.org/pdf/arXiv:1201.3587}.

\bibitem{bib-flagrecent}
R.~Baber and J.~Talbot:
A solution to the $2/3$ conjecture,
{\em SIAM Journal on Discrete Mathematics} {\bf 28} (2014), 756-–766.

\bibitem{bib-flag2}
R.~Baber and J.~Talbot:
Hypergraphs do jump,
{\em Combinatorics, Probability and Computing} {\bf 20} (2011), 161--171.

\bibitem{bib-flag3}
J.~Balogh, P.~Hu, B.~Lidick\'y, and H.~Liu:
Upper bounds on the size of 4- and 6-cycle-free subgraphs of the hypercube,
{\em European Journal on Combinatorics} {\bf 35} (2014), 75--85.

\bibitem{bib-bollobas11+}
B.~Bollob\'as and O.~Riordan:
Sparse graphs: Metrics and random models,
{\em Random Structures \& Algorithms} {\bf 39} (2011), 1--38.

\bibitem{bib-borgs08+}
C.~Borgs, J.T.~Chayes, L.~Lov{\'a}sz, V.T.~S{\'o}s, and K.~Vesztergombi:
Convergent sequences of dense graphs I: Subgraph frequencies, metric properties and testing,
{\em Advances in Mathematics} {\bf 219} (2008), 1801--1851.

\bibitem{bib-borgs+}
C.~Borgs, J.T.~Chayes, L.~Lov\'asz, V.T.~S\'os, and K.~Vesztergombi:
Convergent sequences of dense graphs II. Multiway cuts and statistical physics,
{\em Annals of Mathematics} {\bf 176} (2012), 151--219.

\bibitem{bib-borgs06+}
C.~Borgs, J.~Chayes, L.~Lov{\'a}sz, V.T.~S{\'o}s, B.~Szegedy, and K.~Vesztergombi:
Graph limits and parameter testing,
in: {\em Proceedings of the 38rd Annual ACM Symposium on the Theory of Computing (STOC), ACM, New York, 2006}, 261--270.

\bibitem{bib-chung89+}
F.R.K.~Chung, R.L.~Graham, and R.M.~Wilson:
Quasi-random graphs,
{\em Combinatorica} {\bf 9} (1989), 345--362.

\bibitem{bib-elek07}
G.~Elek:
On limits of finite graphs,
{\em Combinatorica} {\bf 27} (2007), 503--507.

\bibitem{bib-inf}
R.~Glebov, T.~Klimo\v sov\'a, D.~Kr\'al':
{\em Infinite dimensional finitely forcible graphon},
preprint available on \url{http://arxiv.org/pdf/arXiv:1404.2743}.

\bibitem{bib-comp}
R.~Glebov, D.~Kr\'al', J.~Volec:
{\em Compactness and finite forcibility of graphons},
preprint available on \url{http://arxiv.org/pdf/arXiv:1309.6695}.

\bibitem{bib-flag4}
A.~Grzesik:
On the maximum number of five-cycles in a triangle-free graph,
{\em Journal of Combinatorial Theory, Series B}, {\bf 102} (2012), 1061--1066.

\bibitem{bib-flag5}
H.~Hatami, J.~Hladk\'y, D.~Kr\'al', S.~Norine, and A.~Razborov:
Non-three-colorable common graphs exist,
{\em Combinatorics, Probability and Computing} {\bf 21} (2012), 734--742.

\bibitem{bib-flag6}
H.~Hatami, J.~Hladk\'y, D.~Kr\'al', S.~Norine, and A.~Razborov:
On the number of pentagons in triangle-free graphs,
{\em Journal of Combinatorial Theory, Series A}, {\bf 120} (2013), 722--732.

\bibitem{bib-hkmrs1}
C.~Hoppen, Y.~Kohayakawa, C.G.~Moreira, B.~R\'ath, and R.M.~Sampaio:
Limits of permutation sequences,
{\em Journal of Combinatorial Theory, Series B}, {\bf 103} (2013), 93--113.

\bibitem{bib-hkms2}
C.~Hoppen, Y.~Kohayakawa, C.G.~Moreira, and R.M.~Sampaio:
Limits of permutation sequences through permutation regularity,
 preprint available on \url{http://arxiv.org/pdf/arXiv:1106.1663}.

\bibitem{bib-hoppen-test}
C.~Hoppen, Y.~Kohayakawa, C.G.~Moreira, and R.M.~Sampaio:
Testing permutation properties through subpermutations,
{\em Theoretical Computer Science} {\bf 412} (2011), 3555--3567.

\bibitem{bib-janson11}
S.~Janson:
Poset limits and exchangeable random posets,
{\em Combinatorica} {\bf 31} (2011), 529--563.

\bibitem{bib-flag7}
D.~Kr\'al', C.-H.~Liu, J.-S.~Sereni, P.~Whalen, and Z.~Yilma:
A new bound for the 2/3 conjecture,
{\em Combinatorics, Probability and Computing} {\bf 22} (2013), 384--393.

\bibitem{bib-flag8}
D.~Kr\'al', L.~Mach, and J.-S.~Sereni:
A new lower bound based on Gromov’s method of selecting heavily covered points,
{\em Discrete \& Computational Geometry} {\bf 48} (2012), 487--498.

\bibitem{bib-kral12+}
D.~Kr{\'a}l' and O.~Pikhurko:
Quasirandom permutations are characterized by 4-point densities,
{\em Geometric and Functional Analysis} {\bf 23} (2013), 570--579.

\bibitem{bib-lovasz-book}
L.~Lov\'asz:
{\bf Large networks and graph limits},
AMS, Providence, RI, 2012.

\bibitem{bib-lovasz08+}
L.~Lov\'asz and V.T.~S\'os:
Generalized quasirandom graphs,
{\em Journal of Combinatorial Theory, Series B}, {\bf 98} (2008), 146--163.

\bibitem{bib-lovasz11+}
L.~Lov{\'a}sz and B.~Szegedy:
Finitely forcible graphons,
{\em Journal of Combinatorial Theory, Series B}, {\bf 101} (2011), 269--301.

\bibitem{bib-lovasz06+}
L.~Lov{\'a}sz and B.~Szegedy:
Limits of dense graph sequences,
{\em Journal of Combinatorial Theory, Series B}, {\bf 96} (2006), 933--957.

\bibitem{bib-lovasz10+}
L.~Lov{\'a}sz and B.~Szegedy:
Testing properties of graphs and functions,
{\em Israel Journal of Mathematics} {\bf 178} (2010), 113--156.

\bibitem{bib-flag10}
O.~Pikhurko and A.~Razborov:
Asymptotic structure of graphs with the minimum number of triangles,
 preprint available on \url{http://arxiv.org/pdf/arXiv:1204.2846}.

\bibitem{bib-flag9}
O.~Pikhurko and E.R.~Vaughan:
Minimum number of k-cliques in graphs with bounded independence number,
{\em Combinatorics, Probability and Computing} {\bf 22} (2013), 910--934.

\bibitem{bib-wrs}
C.B.~Presutti and W.R.~Stromquist:
Packing rates of measures and a conjecture for the packing density of 2413,
in: {\em Permutation patterns,
London Mathematical Society Lecture Note Series} {\bf 376} (2010), 3--40.

\bibitem{bib-razborov07}
A.~Razborov:
Flag algebras,
{\em Journal of Symbolic Logic} {\bf 72} (2007), 1239--1282.


\bibitem{bib-flag12}
A.~Razborov:
On 3-hypergraphs with forbidden 4-vertex configurations,
{\em SIAM Journal on Discrete Mathematics} {\bf 24} (2010), 946--963.

\bibitem{bib-flag11}
A.~Razborov:
On the minimal density of triangles in graphs,
{\em Combinatorics, Probability and Computing} {\bf 17} (2008), 603--618.

\bibitem{bib-rodl}
V. R\"odl: On universality of graphs with uniformly distributed edges, {\em Discrete Mathematics} {\bf 59} (1986), 125--134.

\bibitem{bib-thomason}
A. Thomason: Pseudo-Random Graphs, in: A. Barlotti, M. Biliotti, A. Cossu, G. Korchmaros and G. Tallini, (eds.), North-Holland Mathematics Studies, North-Holland, 1987, {\bf  144}, 307--331.

\bibitem{bib-thomason2}
A. Thomason: Random graphs, strongly regular graphs and pseudorandom graphs, in: C. Whitehead (ed.), Surveys in Combinatorics 1987, London Mathematical Society Lecture Note Series, 123, Cambridge Univ. Press, Cambridge (1987), 173–-195.

\end{thebibliography}
\end{document}